\documentclass[12pt, reqno]{amsart}

\usepackage{amsmath,amsthm,amsfonts,amscd,amssymb}
\usepackage{graphicx,color,bm,mathtools}
\usepackage[colorlinks=true, citecolor=blue, filecolor=black, linkcolor=black, urlcolor=black]{hyperref}

\usepackage{caption} 
\usepackage{subcaption}

\title[{\fontsize{7pt}{7pt}\selectfont Weingarten calculus for centered random permutation matrices}]
{Weingarten calculus for centered random permutation matrices}
\author {Beno\^\i{}t Collins}
\address{Department of Mathematics, Kyoto University} 
\email{collins@math.kyoto-u.ac.jp}
\author {Manasa Nagatsu}
\address{Department of Mathematics, Kyoto University} 
\email{nagatsu.manasa.64s@st.kyoto-u.ac.jp}

\numberwithin{equation}{section}

\theoremstyle{plain}
\newtheorem{lemma}{Lemma}[section]
\newtheorem{theorem}[lemma]{Theorem}
\newtheorem{proposition}[lemma]{Proposition}
\newtheorem{corollary}[lemma]{Corollary}

\theoremstyle{definition}

\theoremstyle{remark}
\newtheorem{remark}[lemma]{Remark}
\newtheorem{example}[lemma]{Example}

\DeclareMathOperator{\Wg}{Wg}

\renewcommand{\d}{\mathrm{d}}

\begin{document}

\begin{abstract}

We introduce and study the Weingarten calculus for centered random permutation matrices in the symmetric group $S_N$. 
After presenting a formulation of the Weingarten calculus on the symmetric group, we derive a formula in the centered case, as well as a sign-respecting formula.
Our investigations uncover the fact that a building block of this Weingarten calculus is Kummer's confluent hypergeometric function.
It allows us to derive multiple algebraic properties of the Weingarten function and uniform estimate. 
These results shed a conceptual light on phenomena that take place regarding the algebraic and asymptotic behavior of moments of random permutations in the resolution of Bordenave and Bordenave-Collins of strong convergence.
We obtain multiple new non-trivial estimates for moments of coefficients in centered moments.

\end{abstract}

\maketitle

\section{Introduction}

The behavior of random permutations has recently proven to be a key subject in mathematics when it comes to studying generic random graphs. For example, models based on random permutations have been proven by Bordenave to be almost Ramanujan 
\cite{B20_MR4203039}, 
\cite{F08_MR2437174}, and these results have been extended in 
\cite{BC19_MR4024563}
to the random covering of graphs.
In these cases, the proofs rely on moment computations, and it turns out that the correct object to look at is not the random permutation itself (viewed as a unitary $N\times N$ matrix whose entries are 0 or 1) but a centered permutation, acting on the orthogonal space of the Perron Frobenius vector $(1,\ldots, 1)$.

On the other hand, the first named author observed more than a decade ago -- jointly with Ken Dykema, in an unpublished note -- that there exists a Weingarten calculus for random permutations, and this was also described in 
\cite{BCS2012_MR2917777}.

Let us digress here into the Weingarten calculus.
Every compact group $G$ has a unique left- and right-invariant probability measure, known as the Haar measure $\mu_G$.
The Weingarten calculus is a powerful method to study polynomials in random variables in the context of compact matrix groups - more precisely this calculus enables to compute the expectation of a product of entries of Haar-distributed random matrices of the form:
\begin{equation}\label{eq:IntegralOverG}
\int_{g \in G} g_{i_{1} j_{1}} \cdots g_{i_{k} j_{k}} \d \mu_{G} (g).
\end{equation}
Weingarten calculus has wide applications, including random matrix theory, particularly in the context of Haar random unitary matrices, quantum information theory, such as calculations in quantum channels, and free probability theory, where it plays a significant role in studying asymptotic behavior in the large $N$ limit and establishing strong convergence, as discussed in, for example, \cite{CMN_MR4415894}, \cite{CM_MR3680193}, and the references therein.

The theory of Weingarten calculus for the unitary group $\mathcal{U}_N$ or the orthogonal group $\mathcal{O}_N$ is well known. 
However, when it comes to studying random permutations, given that a direct computation of moments turns out to be straightforward, and in that respect, up to now, Weingarten calculus was just an interesting algebraic curiosity. 

However, a critical point of Bordenave's paper \cite{F08_MR2437174}
 is the fact that the product of centered diagonal entries has an expectation that decays fast, typically to a speed comparable to independent variables, as long as one takes a polynomial number of diagonal entries (see the proofs of Proposition 11 and Lemma 12 of \cite{B20_MR4203039}). 
In this paper, we endeavor to study this phenomenon of fast decay more systematically and obtain the following results.

 (1) We present a new formulation for permutations (Theorem \ref{thm:non-CenteredWg}) and for
 centered random permutation matrices (Theorem \ref{thm:CenteredWg1})
 
 (2) A Collins-Matsumoto-Novak-type formula (\cite{CM_MR3680193}, \cite{MN_MR3010693}) 
 for the Weingarten coefficient (Theorem \ref{thm:CenteredWg2}).
 Specifically, the original Weingarten coefficients for the unitary or orthogonal Haar measures had generating functions in the dimension of the group that involved signed functions
(\cite{C2003_MR1959915}, 
\cite{CS06_MR2217291}), 
and later we managed to remove this. For quantum groups, the first named author did the same in a joint work with Brannan in \cite{BC_MR3874001}.
We uncover the fact that the building block of the centered Weingarten calculus for permutations is a quantity that we call $a_k(N)$, which is closely related to Kummer's confluent hypergeometric functions.

(3)
In particular, this gives algebraic results about the sign of Weingarten functions and shows that a quantity introduced by Bordenave (Section \ref{section-charles-element}) is conceptually the correct object to estimate the centered Weingarten function. 

(4)
Uniform estimates for the Weingarten functions of centered permutations. 
In particular, this gives new non-trivial estimates for moments of coefficients in centered moments. 
An initial motivation to this problem was described to the second author by Bordenave, in the hope of checking to which extent one can compare the moments of i.i.d.-centered Bernoulli and random permutations. 

This paper is organized as follows.
After this introduction, in Section \ref{sec:Non-CenteredWg}, we introduce notations on permutation matrices obtained from the symmetric group $S_N$, and formulate the Weingarten calculus for them.
In Section \ref{sec:CenteredWg}, we formulate the main results of the paper: the Weingarten calculus for centered random permutation matrices.
In Section \ref{sec:Asymptotics}, we illustrate the asymptotics of the Weingarten coefficients defined in Section \ref{sec:CenteredWg}.  
In the final sections, we discuss a recursive reduction to compute integrals such as \eqref{eq:IntegralOverG} for centered random permutation matrices in Section \ref{sec:reduction}, and we also specifically discuss specific examples of the failure for the uniform bound of the Weingarten coefficients in Section \ref{sec:FailureUniformBound}.
In Section \ref{sec:moment}, we discuss specific examples that describe how those techniques help the understanding of the moment integral.

\subsection*{Acknowledgements}
B. C. is supported by JSPS Grant-in-Aid Scientific Research (B) no. 21H00987, and Challenging Research (Exploratory) no.  23K17299.
Part of this work emerged from discussions with Ken Dykema around fifteen years ago. We thank him for the discussions at that time. 

M. N. was partially supported by Overseas Dispatch Expenses for Students from the Super Global Course 
and by the Mathematics and Mathematical Sciences Innovation Course, Kyoto University.
Part of this research project was conducted during her stay at ENS Lyon, and she would like to acknowledge the hospitality of ENS Lyon and, in particular, Prof. Alice Guionnet.
The research was inspired by discussions with Prof. Charles Bordenave and she would like to thank him and Chaire Jean Morlet at CIRM for their hospitality. 

Finally, both authors are indebted to Sho Matsumoto for a careful reading of a preliminary version of this manuscript and very useful feedback.

\section{Permutation Matrices}\label{sec:Non-CenteredWg}

\subsection{Notations}
Throughout this paper, we consider the normalized Haar integration over the symmetric group $S_N$. Here, the normalized Haar measure is the counting measure on $S_N$ renormalized by a factor $1/N!$. Any element $g \in S_N$ can be identified as $N\times N$ matrix called {\em permutation matrix} with exactly one $1$ in each row and column, that is, $(r,s)$-element of the matrix $g$ is $1$ if and only if $g(r)=s$ is a permutation.

For details of the following notations on set partitions, we refer to \cite{S_MR1442260}. If $Z$ is a finite set, a partition of $Z$ is a set $\pi=\left\{B_{1}, \ldots, B_{r}\right\}$ of pairwise disjoint non-empty subsets $B_{j}$ of $Z$ whose union is all of $Z$. These sets $B_{j}$ are called the blocks of the partition, and the number of blocks of $\pi$ is simply denoted $\# \pi$. 

The set of partitions of the set of $k$ integers $[k] :=\{1, \cdots, k\}$, denoted by $\mathcal{P}(k)$, forms a lattice with respect to the partial order, the meet operation $\wedge$ and the join operation $\vee$ are defined as follows.  
\begin{itemize}
    \item partial order $\leq $: \ For $\pi_1, \pi_2 \in \mathcal{P}(k)$, we call $\pi_1 \leq \pi_2 $ if and only if each block of $\pi_1$ is contained in some block of $\pi_2$. The maximum partition is denoted by $1_k = \{\{1,\cdots , k\}\}$, and the minimum partition is denoted by $0_k=\{\{1\}, \cdots , \{k\}\}$ .
    \item meet operation $\wedge$: \ For $\pi_1, \pi_2 \in \mathcal{P}(k)$, $\pi_1 \wedge \pi_2$ is the greatest lower bound of $\pi_1$ and $\pi_2$.
    \item join operation $\vee$: For $\pi_1, \pi_2 \in \mathcal{P}(k)$, $\pi_1 \vee \pi_2$ is the least upper bound of $\pi_1$ and $\pi_2$.
\end{itemize}

We define incidence algebra on the lattice $\mathcal{P}(k)$,
\[I(\mathcal{P}(k))=\{f:\mathcal{P}(k) \times \mathcal{P}(k) \to \mathbb{C}\ | \ f(\pi_1,\pi_2)=0 \mbox{ if } \pi_{1} \nleq \pi_{2}\}\]
with an associative convolution *, for any $\pi_1, \pi_2 \in \mathcal{P}(k)$,
\begin{equation*}
f * g\left(\pi_{1}, \pi_{2}\right):=\sum_{\pi_{3}: \pi_{1} \leq \pi_{3} \leq \pi_{2}} f\left(\pi_{1}, \pi_{3}\right) g\left(\pi_{3}, \pi_{2}\right). 
\end{equation*}
$I(\mathcal{P}(k))$ has the unit $\delta$, Kronecker delta, 
\[
\delta(\pi_1, \pi_2) 
= \begin{cases}
    1 & \mbox{if } \pi_1 = \pi_2,\\
    0 & \mbox{otherwise}.
\end{cases}
\]

We define the Zeta function $\zeta$ by 
\begin{equation*}
\zeta\left(\pi_{1}, \pi_{2}\right)
= \begin{cases}
    1 & \mbox{if }\pi_{1} \leq \pi_{2},\\
    0 & \mbox{otherwise},
\end{cases}
\end{equation*}
for $\pi_1, \pi_2 \in \mathcal{P}(k)$.

We also introduce the Möbius function $\mu$. We first define \[ \mu(0_k, 1_k)=(-1)^{k-1} (k-1)!. \] 
For any $\pi_{1}, \pi_{2} \in \mathcal{P}(k)$, when we suppose that $\pi_2 = \{B_1, \dots , B_{\# \pi_2}\}$ and that each block $B_i$ of $\pi_2$ is partitioned into $\lambda_i$ blocks in $\pi_1$, the interval $\left[\pi_{1}, \pi_{2}\right]$ in the lattice $\mathcal{P}(k)$ is isomorphic to
$\left[0_{\lambda_1}, 1_{\lambda_1}\right] \times \cdots \times [0_{\lambda_{\# \pi_2}}, 1_{\lambda_{\# \pi_2}}]$. Now we define  
$$
\mu\left(\pi_{1}, \pi_{2}\right)
=\prod_{i=1}^{\# \pi_2} \mu(0_{\lambda_i}, 1_{\lambda_i})
=\prod_{i=1}^{\# \pi_2}(-1)^{\lambda_i-1}(\lambda_i-1)!.
$$

We need the fact 
that the convolution product of the Möbius function $\mu$ and the Zeta function $\zeta$ is equal to the Kronecker delta $\delta$,
\begin{equation}\label{MöbiusInversion}
\mu * \zeta=\delta.
\end{equation}
We call this {\em Möbius inversion formula}. 
(See \cite{S_MR1442260} Section 3.7)

Finally, we use the multi-index notation $\mathbf{i}=(i_1,\dots,i_k), \ \mathbf{j}=(j_1,\dots,j_k)$.
For a multi-index $\mathbf{i}= (i_1,...,i_k) \in \{1, \cdots , N\}^k$, we view $\mathbf{i}$ as a level partition $\Pi_\mathbf{i} \in \mathcal{P}(k)$ so that $1 \leq l,m \leq k$ belongs to the same block of $\Pi_\mathbf{i}$ if and only if $i_l=i_m$.

\subsection{Weingarten calculus for Random Permutation Matrices}

We first consider the Weingarten calculus for random permutation matrices. Here, the symmetric group $S_N$ is identified as the group of permutation matrices. 

In the sequel, we use the Pochhammer symbol notation, 
\[(x)_n:=x(x-1)\cdots (x-n+1) = \frac{x!}{(x-n)!}.\]
Our first result is the Weingarten formula for random permutations. Note that a variant of this theorem appears in \cite{BCS2012_MR2917777} and was part of an unpublished project between the first-named author and Ken Dykema. 

\begin{theorem}\label{thm:non-CenteredWg}
For any $k \geq 1$ and given indices $\mathbf{i}=(i_1,\dots,i_k)$ and $\mathbf{j}=(j_1,\dots,j_k)$, where $i_1,\dots,i_k, j_1, \dots,j_k \in \{1,2, \ldots, N\}$,
there exists some Weingarten function $\Wg_k$ so that 
\begin{equation}
\int_{g \in S_{N}} g_{i_{1} j_{1}} \cdots g_{i_{k} j_{k}} \d g 
=\sum_{\sigma, \tau \in \mathcal{P}(k)} \Wg_k(\sigma, \tau, N) \zeta(\sigma, \Pi_\mathbf{i}) \zeta(\tau, \Pi_\mathbf{j}).
\end{equation}
(Note that the integral $\d g$ runs over the normalized Haar measure on $S_N$.)\ Moreover, we have the following explicit formula for $\Wg_k(\sigma, \tau, N)$:
\begin{equation}\label{eq:def_WgPerm}
\begin{split}
\Wg_k(\sigma, \tau, N)
&=\sum_{\pi \leq \sigma \wedge \tau \in \mathcal{P}(k)} \mu\left(\pi, \sigma \right) \mu\left(\pi, \tau\right) \frac{1}{(N)_{\# \pi}}.
\end{split}
\end{equation}
\end{theorem}

\begin{proof}
One sees by inspection that we have
\begin{equation}\label{eq:ExpRandPerm}
\int_{g \in S_N} g_{i_{1} j_{1}} \cdots g_{i_{k} j_{k}} \d g
=\delta (\Pi_{\mathbf{i}}, \Pi_{\mathbf{j}}) \frac{1}{(N)_{\# \Pi_{\mathbf{i}}}}.
\end{equation}

If we define a function $W$ on $\mathcal{P}(k)^2$ as
\[
W(\sigma, \tau) 
:=\sum_{\sigma^\prime \leq \sigma, \tau^\prime \leq \tau \in \mathcal{P}(k)} \mu(\sigma^\prime, \sigma) \mu(\tau^\prime, \tau) \delta(\sigma^\prime, \tau^\prime) \frac{1}{(N)_{\# \sigma^\prime}}
\]
for $\sigma, \tau \in \mathcal{P}(k)$, then it can be inverted by Möbius inversion formula (\ref{MöbiusInversion}), for any $\Pi_{\mathbf{i}}, \Pi_{\mathbf{j}} \in \mathcal{P}(k)$,
\[
\begin{split}
&\sum_{\sigma, \tau \in \mathcal{P}(k)} W(\sigma, \tau) \zeta(\sigma, \Pi_\mathbf{i}) \zeta(\tau, \Pi_\mathbf{j})\\
=&\sum_{\sigma, \tau \in \mathcal{P}(k)} \sum_{\substack{\sigma^\prime \leq \sigma, \\ \tau^\prime \leq \tau}} \mu(\sigma^\prime, \sigma) \mu(\tau^\prime, \tau) \delta(\sigma^\prime, \tau^\prime) \frac{1}{(N)_{\# \sigma^\prime}} \zeta(\sigma, \Pi_\mathbf{i}) \zeta(\tau, \Pi_\mathbf{j})\\
=& \sum_{\sigma^\prime= \tau^\prime \in \mathcal{P}(k)} 
\sum_{\sigma: \sigma^\prime \leq \sigma \leq \Pi_{\mathbf{i}}} \mu(\sigma^\prime, \sigma)\zeta(\sigma, \Pi_\mathbf{i}) 
\sum_{\tau: \tau^\prime \leq \tau \leq \Pi_{\mathbf{j}}} \mu(\tau^\prime, \tau)\zeta(\tau, \Pi_\mathbf{j}) 
\frac{1}{(N)_{\# \sigma^\prime}}\\
=& \sum_{\sigma^\prime= \tau^\prime \in \mathcal{P}(k)} \delta(\sigma^\prime, \Pi_{\mathbf{i}}) \delta(\tau^\prime, \Pi_{\mathbf{j}})\frac{1}{(N)_{\# \sigma^\prime}}\\
=& \ \delta (\Pi_{\mathbf{i}}, \Pi_{\mathbf{j}}) \frac{1}{(N)_{\# \Pi_{\mathbf{i}}}}.
\end{split}
\]
Thus, combining with equation (\ref{eq:ExpRandPerm}) implies that the function $W$ satisfies
\[
\int_{g \in S_N} g_{i_{1} j_{1}} \cdots g_{i_{k} j_{k}} \d g
=\sum_{\sigma, \tau \in \mathcal{P}(k)} W(\sigma, \tau) \zeta(\sigma, \Pi_\mathbf{i}) \zeta(\tau, \Pi_\mathbf{j}),
\]
and this ends the proof.
\end{proof}

Note that it is unusual in Weingarten calculus that there exists such a simple closed formula for $\Wg$. This is due to the fact that $\int_{g \in S_{N}} g_{i_{1} j_{1}} \cdots g_{i_{k} j_{k}} \d g$ can be computed by different means, which is a difficult task for other compact groups.

As a corollary, we obtain the following estimate of the Weingarten function:

\begin{corollary}
Every entry of the Weingarten matrix $\Wg_k$ is a polynomial in $N$ divided by 
$(N)_k=N(N-1)\cdots (N-k+1)$.    
\end{corollary}

Note that similar results are known for other types of Weingarten calculus (e.g., unitary \cite{C2003_MR1959915}, orthogonal \cite{CM_MR3680193}, quantum unitary \cite{BC2007_MR2341011}), but in each case, the multiplication factor
$(N)_k$
has to be replaced by a polynomial of a much larger degree, and the techniques of proof are much more involved.

Next, we give estimates on the asymptotic behavior of $\Wg_k$ as $N$ is large.
The following corollary evaluates the asymptotics of the $S_n$ Weingarten function.
It is of independent algebraic interest and will also serve as a benchmark to compare the gain in decay one achieves by centering random permutations. 
\begin{corollary}\label{cor:Asym_WgPerm}
The Weingarten function as described above satisfies
\begin{equation}\label{eq:asym_WgPerm}
\Wg_k(\sigma, \tau, N)=\mu(\sigma \wedge \tau, \sigma) \mu(\sigma \wedge \tau, \tau) N^{-\# (\sigma \wedge \tau)}(1+O(N^{-1})). 
\end{equation}
\end{corollary}

\begin{proof}
This follows directly from the closed formula and the fact that the only leading term in the sum is $\sigma \wedge \tau$.
\end{proof}

Note that it was proved in Proposition 3.4, of \cite{BCS2012_MR2917777} that
\begin{equation}\label{eq:asym_BCS2012}
\Wg_k(\sigma, \tau, N)=O\left(N^{\#(\sigma \vee \tau)-\# \sigma-\# \tau}\right)
\end{equation}
holds. This is consistent with our estimate $O\left(N^{-\# (\sigma \wedge \tau)}\right)$, since by semi-modularity of $\mathcal{P}(k)$ we have
\begin{equation}\label{ineq:MeetJoint}
-\# (\sigma \wedge \tau) \leq\#(\sigma \vee \tau)-\#\sigma-\# \tau 
\end{equation}
for any $\sigma, \tau \in \mathcal{P}(k)$. However, in many cases, the equality does not hold in (\ref{ineq:MeetJoint}), and the precise asymptotics of (\ref{eq:asym_WgPerm}) are strictly stronger than the estimate (\ref{eq:asym_BCS2012}). However, note that our result is specific to the group $S_{N}$, whereas the result of \cite{BCS2012_MR2917777} holds for more general easy quantum symmetric groups. 

\section{Centered Permutation Matrices}\label{sec:CenteredWg}

\subsection{Weingarten calculus for Centered Random Permutation Matrices}

We now consider the {\em centered} random permutation matrices. Throughout this paper, a {\em centered random permutation matrix} obtained from $g \in S_N$, denoted by $[g]$, is defined as
\[
[g]= g - \int_{h \in S_N} h \d h.
\]
What will matter mainly to us is that $[g]_{i j}=[g_{i j}] = g_{i j}-\frac{1}{N}$.

Before we formulate the Weingarten formula for centered random permutation matrices, we introduce the following notation, which will be used throughout this paper. 
For $\pi \in \mathcal{P}(k)$, we denote by
\[D(\pi)= \{j \in [k](=\{1,\cdots ,k\})\ | \ \mbox{the singleton }\{j\} \mbox{ is a block of } \pi \}\]
the singleton set.
For $\pi \in \mathcal{P}(k)$ and $M \subset [k]$, the restriction of $\pi$ to the set $M$ is denoted by $\pi|_M$. 

In the following theorem, we obtain the Weingarten formula for centered random permutation matrices:

\begin{theorem}\label{thm:CenteredWg1}
For any $k \geq 1$ and given indices $\mathbf{i}=(i_1,\dots,i_k)$ and $\mathbf{j}=(j_1,\dots,j_k)$, where $i_1,\dots,i_k, j_1, \dots,j_k \in \{1,2, \ldots, N\}$,
\begin{equation}
\int_{g \in S_{N}} [g_{i_{1} j_{1}}] \cdots [g_{i_{k} j_{k}}] \d g 
=\sum_{\sigma, \tau \in \mathcal{P}(k)} \mathring{\Wg}_k(\sigma,\tau,N) \ \zeta(\sigma, \Pi_\mathbf{i}) \zeta(\tau, \Pi_\mathbf{j}), 
\end{equation}
where 
\begin{equation}\label{eq:Def_CenteredWg}
\begin{split}
&\mathring{\Wg}_k(\sigma, \tau, N)\\
=& \sum_{i = 0}^{ |D(\sigma \vee \tau)|} \tbinom{ |D(\sigma \vee \tau)|}{i} (-1)^i \sum_{\pi \leq \sigma \wedge \tau \in \mathcal{P}(k)} \mu(\pi, \sigma)\mu(\pi,\tau) N^{-i} \ 
\frac{1}{(N)_{\# \pi -i}}.
\end{split}
\end{equation}
\end{theorem}

\begin{proof}
By expanding, using $[g_{i j}] = g_{i j}-\frac{1}{N}$ and applying Theorem\ref{thm:non-CenteredWg}, we have 
\[
\begin{split}
&\int_{g \in S_{N}} [g_{i_{1} j_{1}}] \cdots [g_{i_{k} j_{k}}] \d g \\
=& \sum_{M \subseteq [k]} \left(- \frac{1}{N}\right)^{k-|M|} 
\int_{g \in S_{N}} \left( \prod_{m \in M} g_{i_m j_m}\right) \d g \\
=& \sum_{M \subseteq [k]} \left(- \frac{1}{N}\right)^{k-|M|} 
\sum_{\sigma^\prime,\tau^\prime \in \mathcal{P}(|M|)} 
\Wg_{|M|}(\sigma^\prime, \tau^\prime, N) \zeta(\sigma^\prime, \Pi_{\mathbf{i}}|_M) \zeta(\tau^\prime, \Pi_{\mathbf{j}}|_M). 
\end{split}
\]
Here, we note that there is a natural one-to-one correspondence $\mathcal{P}(|M|) \to \{\sigma \in \mathcal{P}(k) \ |\ \sigma|_{[k]\setminus M}=0_{k-|M|} \}$ that maps $\sigma^\prime \mapsto \sigma$ so that
$\sigma|_M =\sigma^\prime, \ \sigma|_{[k]\setminus M}=0_{k-|M|}$. Thus, 
\[
\begin{split}
&\sum_{\sigma^\prime,\tau^\prime \in \mathcal{P}(|M|)} \Wg_{|M|}(\sigma^\prime, \tau^\prime, N) \zeta(\sigma^\prime, \Pi_{\mathbf{i}}|_M) \zeta(\tau^\prime, \Pi_{\mathbf{j}}|_M)\\
&=\sum_{\substack{\sigma,\tau \in \mathcal{P}(k) \text{ s.t.} \\ \sigma|_{[k] \backslash M} = \tau|_{[k]\backslash M} = 0_{k-|M|}}} \Wg_{|M|}(\sigma|_M, \tau|_M, N) \zeta(\sigma, \Pi_{\mathbf{i}}) \zeta(\tau, \Pi_{\mathbf{j}}).
\end{split}
\]
This yields 
\[
\begin{split}
&\int_{g \in S_{N}} [g_{i_{1} j_{1}}] \cdots [g_{i_{k} j_{k}}] \d g \\
=& \sum_{\sigma,\tau \in \mathcal{P}(k)} \left( \sum_{\substack{M \subseteq [k] \text{ s.t.} \\ [k]\backslash M \subseteq D(\sigma \vee \tau)}}
\Wg_{|M|}(\sigma|_M, \tau|_M, N) \left(- \frac{1}{N}\right)^{k-|M|} \right) \zeta(\sigma, \Pi_{\mathbf{i}}) \zeta(\tau, \Pi_{\mathbf{j}})\\
=& \sum_{\sigma, \tau \in \mathcal{P}(k)} \mathring{\Wg}_k(\sigma,\tau, N) \ \zeta(\sigma, \Pi_{\mathbf{i}}) \zeta(\tau, \Pi_{\mathbf{j}}),
\end{split}
\]

where
\[
\begin{split}
&\mathring{\Wg}_k(\sigma,\tau, N)
:=\sum_{\substack{M \subseteq [k] \text{ s.t.}\\ [k]\backslash M \subseteq D(\sigma \vee \tau)}}
\Wg_{|M|}(\sigma|_M, \tau|_M, N) \left(- \frac{1}{N}\right)^{k-|M|} \\
&=\sum_{i = 0}^{ |D(\sigma \vee \tau)|} \tbinom{ |D(\sigma \vee \tau)|}{i} (-1)^i \sum_{\pi \leq \sigma \wedge \tau \in \mathcal{P}(k)} \mu(\pi, \sigma)\mu(\pi,\tau) N^{-i} \ \frac{1}{(N)_{\# \pi -i}}.
\end{split}
\]
\end{proof}

\begin{remark}
    The form of $\mathring{\Wg}_k(\sigma, \tau, N)$ in Theorem \ref{thm:CenteredWg1} does not give us much information about it. For example, one cannot deduce its sign or its leading order as $N \to \infty$.
    The sequel of this paper aims to address this question. 
\end{remark}

\subsection{A specific coefficient of the Weingarten matrix}\label{section-charles-element}
An important drawback of equation \eqref{eq:Def_CenteredWg} is that it involves signed coefficients. 
We want to obtain an improved -- that is, nonsigned ---- formulation of $\mathring{\Wg}_k$ to enable us to study its algebraic and asymptotic properties. 
This situation can be compared with the orthogonal or unitary Weingarten function: the original formula of \cite{C2003_MR1959915} was signed, whereas the subsequent formulations by \cite{CMN_MR4415894} are not signed. 

In the case of orthogonal or unitary Weingarten calculus, Jucys Murphy elements and special variants of Cayley graphs on the symmetric groups were needed. 
For the purpose of the centered symmetric random matrices, it turns out that an important quantity plays a role, which we introduce and study in this section.

First, we consider the following special entry of $\mathring{\Wg}_k$:
\begin{equation}\label{eq:Def_a_k}
a_k (N):=\mathring{\Wg}_k(0_k, 0_k, N)
=\int_{g \in S_N} [g_{1 1}]\cdots [g_{k k}] \d g.
\end{equation}
Let us introduce the following notation: $n!!:=\prod_{m=1}^{(n+1)/2}(2m-1)$ for a positive odd integer $n$, and $(-1)!!=1$.
The quantity $a_k(N)=\mathring{\Wg}_k(0_k, 0_k, N)$ can be computed explicitly as follows:
\begin{proposition}\label{prop:a_k}
The quantity $a_k (N)$ satisfies the equation
\begin{equation}\label{eq:a_k}
a_k (N)
=\sum_{l=0}^k \tbinom{k}{l} 
\tfrac{1}{(N)_l}
(-N)^{-(k-l)}.
\end{equation}

In particular, $a_k(N)>0$ for all $k \geq 2$ and $a_0 (N)=1, a_1(N)=0$.

Moreover, if $N \gg k^3$ and $k \geq 1$, then
\begin{equation}\label{eq:UniformBound_a_k}
a_k(N) = \tilde{a}_k(N)\ (1+O(k^3/N)),
\end{equation}
where 
\begin{equation}\label{eq:Def_tilde_a_k}
\tilde{a}_k(N) =
\begin{cases}
    (2k^\prime-1)!! N^{-3k^\prime} &\mbox{if }\ k=2k^\prime\\
    \frac{4k^\prime}{3} (2k^\prime+1)!! N^{-3k^\prime-2} & \mbox{if }\ k=2k^\prime+1.
\end{cases}
\end{equation}

\end{proposition}

\begin{proof}
(\ref{eq:a_k}) can be directly obtained from its definition (\ref{eq:Def_a_k}) by applying (\ref{eq:ExpRandPerm}),
\begin{equation}\label{eq:alg-def-akn}
    \begin{split}
a_k(N) 
&= \int_{g \in S_N} [g_{1 1}]\cdots [g_{k k}] \d g\\
&=\sum_{M \subseteq [k]} \left(- \frac{1}{N}\right)^{k-|M|} 
\int_{g \in S_{N}} \left( \prod_{m \in M} g_{m m}\right) \d g\\
&=\sum_{l=0}^k \tbinom{k}{l} (-N)^{-(k-l)} \tfrac{1}{(N)_l}.
\end{split}
\end{equation}

It remains to prove the asymptotic behavior of $a_k(N)$ as $N \to \infty$ and the uniform bound.
We first get a recurrence formula of $a_k :=a_k (N)$. 
For $k \geq 2$, $\sum_{j=1}^N [g_{k\ j}] = 0$ yields
\[
\begin{split}
&0=\sum_{j=1}^N \int_{g \in S_N} [g_{1 \ 1}]\cdots [g_{k-1 \ k-1}][g_{k \ j}] \d g \\
&=(k-1)\int_{g \in S_N} [g_{1\ 1}]\cdots [g_{k-1 \ k-1}][g_{k \ k-1}] \d g\\
&\qquad 
+ (N-k+1) \int_{g \in S_N} [g_{1\ 1}]\cdots [g_{k-1 \ k-1}] [g_{k \ k}] \d g.
\end{split}
\]
Here, we note that $\int_{g \in S_N} [g_{1\ 1}]\cdots [g_{k-1 \ k-1}] [g_{k \ k}] \d g=a_k$, and $g_{k-1 \ k-1} g_{k \ k-1} =0$, which yields
\[
\begin{split}
[g_{k-1 \ k-1}][g_{k \ k-1}]
&= (g_{k-1 \ k-1}-\tfrac{1}{N})(g_{k \ k-1}-\tfrac{1}{N})\\
&=-\tfrac{1}{N}[g_{k-1 \ k-1}]-\tfrac{1}{N}[g_{k \ k-1}]-\tfrac{1}{N^2}.
\end{split}
\]
Thus, 
\[
\begin{split}
0
&=(k-1)\int_{g \in S_N} [g_{1\ 1}]\cdots [g_{k-2 \ k-2}](-\tfrac{1}{N}[g_{k-1 \ k-1}]-\tfrac{1}{N}[g_{k \ k-1}]-\tfrac{1}{N^2}) \d g \\
& \qquad +(N-k+1)a_k\\
&=(k-1)(-\frac{2}{N} a_{k-1} -\frac{1}{N^2}a_{k-2}) +(N-k+1)a_k.
\end{split}
\] 
Thus, we can obtain a recursion
\begin{equation}\label{eq:recursion_a_k}
a_k - \frac{2(k-1)}{N(N-k+1)} a_{k-1} - \frac{k-1}{N^2 (N-k+1)} a_{k-2} =0.
\end{equation}
We note that this recursive equation ensures by induction that $a_k(N)$ is always nonnegative.

Now, we define $\varepsilon_k := \frac{a_k- \tilde{a_k}}{\tilde{a_k}}$ (that is, $a_k = \tilde{a_k} (1+ \varepsilon_k)$).
It suffices to show that for all $k$,
$
\varepsilon_k = O(\frac{k^3}{N}) .
$
We will compute positive constants $C_1 > 1, C_2 > 1$ so that
for any positive integers $k$ and $N$, if $N>C_1 k^3$, then 
\[
|\varepsilon_k|<\frac{C_2 k^3}{N}.
\]

(i)
For $k=2l$, substituting $a_k = \tilde{a_k} (1+ \varepsilon_k)$ into the recurrence, we get
\[
\varepsilon_{2l}
=\frac{4(2l-1)(2l-2)}{3(N-2l+1)} \varepsilon_{2l-1} + \frac{N}{N-2l+1}\varepsilon_{2l-2} + \frac{(8l-5)(2l-1)}{3(N-2l+1)},
\]
so if $k$ is even, 
\[
|\varepsilon_k|< \frac{4k^2}{3(N-k+1)}|\varepsilon_{k-1}|+\frac{N}{N-k+1}|\varepsilon_{k-2}|+\frac{4k^2}{3(N-k+1)}.
\]
(ii)
For $k=2l+1$, in the same way as in (i),
\[
\varepsilon_{2l+1}=\frac{3N}{(N-2l)(2l+1)} \varepsilon_{2l} + \frac{N(2l-2)}{(N-2l)(2l+1)}\varepsilon_{2l-1} +\frac{2l}{N-2l},
\]
so if $k$ is odd,
\[
|\varepsilon_k|< \frac{3N}{(N-k+1)k}|\varepsilon_{k-1}|+\frac{N}{N-k+1}|\varepsilon_{k-2}|+\frac{k}{N-k+1}.
\]

Here we note that if $N>C_1 k^3$, then 
\[
\frac{1}{N-k+1} < \frac{1}{N-k} 
=\frac{1}{N} (1+\frac{k}{N-k})
< \frac{1}{N} (1+\frac{1}{C_1 -1} \frac{1}{k^2}),
\]
and 
\[\frac{4}{3 }k^2 < \frac{4N}{3C_1 k}<\frac{3N}{k}.\]

Therefore, in either case that $k$ is even or odd,
\[
|\varepsilon_{k}|
< (1+\frac{1}{C_1 -1} \frac{1}{k^2})(\frac{3}{k}|\varepsilon_{k-1}|+|\varepsilon_{k-2}|+\frac{4k^2}{3N}).
\]
So when $N>C_1 k^3$, if we assume that 
$|\varepsilon_{k-1}|<\frac{C_2 (k-1)^3}{N}, |\varepsilon_{k-2}| <\frac{C_2 (k-2)^3}{N}$
, then
\[
|\varepsilon_{k}|
< (1+\frac{1}{C_1 -1} \frac{1}{k^2})\frac{C_2}{N}(\frac{3(k-1)^3}{k} + (k-2)^3 + \frac{4k^2}{3C_2}).
\]

When we put $C_1 = C_2 =2$, they satisfy, for all $k$,
\[(1+\frac{1}{C_1 -1} \frac{1}{k^2})\frac{C_2}{N}(\frac{3(k-1)^3}{k} + (k-2)^3 + \frac{4k^2}{3C_2}) <\frac{C_2 k^3}{N}.\]
Therefore, by combining with the fact 
$|\varepsilon_1|=0, |\varepsilon_2|=\frac{1}{N-1}<\frac{2\cdot 2^3}{N}$,
one can prove inductively that 
for $C_1 = C_2 =2$ and for any positive integers $k, N$ with $N>C_1 k^3$,  
\[
|\varepsilon_k|<\frac{C_2 k^3}{N}.
\]

\end{proof}

\begin{remark}
Inspecting the proof, one observes that it is not possible to push the argument beyond the case $k^3\le C N$. It would be interesting to check whether this is the optimal scaling. This question does not seem obvious to us and we leave it for future work.     
\end{remark}

\begin{remark}
The quantity $a_k(N)$ has an interpretation in terms of Kummer's confluent hypergeometric function $M(a,b,z)$, we refer to section 13 of \cite{AS_MR208797}, defined as 
$$
M(a, b, z):=1+\frac{a z}{b}+\frac{(a)^2}{(b)^2} \frac{z^2}{2!}+\cdots+\frac{(a)^n}{(b)^n}\frac{z^n}{n!}+\cdots,
$$
where
$$
(a)^0:=1,(a)^n:=a(a+1) \cdots(a+n-1).
$$
By using this notation, we get
\[
a_k(N)=\left(-\frac{1}{N}\right)^k M(-k,-N,-N).
\]
After completing the above proof, we noticed that
the quantity $M(a,b,z)$ satisfies a recursive relation, (13.4.1) of \cite{AS_MR208797}: 
$$(b-a) M(a-1, b, z)+(2 a-b+z) M(a, b, z)-a M(a+1, b, z)=0, $$
which allows us to recover the recursive equation \eqref{eq:recursion_a_k} of $a_k(N)$.
We leave our proof 
for the sake of self-containedness and because of its probabilistic insights and use of symmetries of the symmetric group which do not appear in \cite{AS_MR208797}.
\end{remark}

\subsection{Reformulation of the Weingarten function}

Now we can reformulate $\mathring{\Wg}$ in terms of the sequence $\{a_k (N)\}_k$ defined in Proposition \ref{prop:a_k}.

\begin{theorem}\label{thm:CenteredWg2}
For any $\sigma, \tau \in \mathcal{P}(k)$, 
\begin{equation}\label{eq:CenteredWg_improved}
\begin{split}
&\mathring{\Wg}_k(\sigma, \tau, N)\\
=&\sum_{j=\#(\sigma \wedge \tau)}^{k} 
\left( \sum_{\substack{\pi \leq \sigma \wedge \tau \\ \#\pi=j }}^{} \mu(\pi, \sigma) \mu(\pi, \tau) \right)
\sum_{i=0}^{j-|D(\sigma \vee \tau)|}
\tbinom{j-|D(\sigma \vee \tau)|}{i}
N^{-i}
a_{j-i}.
\end{split}    
\end{equation}
\end{theorem}
Before we prove this theorem, let us mention a nontrivial corollary that cannot be obtained from the initial Weingarten formula of Theorem \ref{thm:CenteredWg1}:

\begin{corollary}
    For any $N,\sigma,\tau$, the sign of $\mathring{\Wg}_k(\sigma, \tau, N)$ is $(-1)^{\# \sigma + \# \tau}$.
\end{corollary}
This follows from the fact that the sign of $\mu(\pi, \sigma) \mu(\pi, \tau)$ equals $(-1)^{\# \sigma + \# \tau}$, which is independent of $\pi \leq \sigma \wedge \tau$.
Let us now prove Theorem \ref{thm:CenteredWg2}

\begin{proof}
By comparing with the formula (\ref{eq:Def_CenteredWg})
with 
$p := \# \pi = k- |\pi| , d := |D(\sigma \vee \tau)|$, 
it suffices to show that for any nonnegative integers $d,p \in \mathbb{Z}_{\geq 0}$ with $0 \leq d\leq p \leq k$,
\begin{equation}\label{3.9}
\begin{split}
\sum_{i_1 = 0}^{p-d} \binom{p-d}{i_1} N^{-i_1} a_{p-i_1} =
\sum_{i_2 =0}^{d} \binom{d}{i_2} \frac{(N-p+i_2)!}{N!} (-N)^{-i_2}.
\end{split}
\end{equation}

This is shown by calculation using a formula for binomial coefficients, stated later in the Lemma \ref{lemma:CombiIncExc}.
The Left-Hand side of equation \eqref{3.9} is
\[
\begin{split}
&\sum_{i_1=0}^{p-d}\binom{p-d}{i_1} N^{-i_1} \sum_{l=0}^{p-i_1}\binom{p-i_1}{l} \frac{(N-l)!}{N!}\left(-N^{-1}\right)^{p-i_1-l} \\
& =\sum_{i_1=0}^{p-d}\binom{p-d}{i_1} \sum_{i_2=i_1}^p\binom{p-i_1}{p-i_2} \frac{\left(N-p+i_2\right)!}{N!}(-1)^{i_2-i_1} N^{-i_2} \\
& \qquad (\mbox{change of variables, } i_2  =p-l)\\
& =\sum_{i_2=0}^p \frac{\left(N-p+i_2\right)!}{N!}(-N)^{-i_2} \sum_{i_1=0}^{\min \left\{p-d, i_2\right\}}(-1)^{i_1}\binom{p-d}{i_1}\binom{p-i_1}{p-i_2} \\
& =\sum_{i_2=0}^p \frac{\left(N-p+i_2\right)!}{N!}(-N)^{-i_2}\binom{d}{d-i_2} \chi_{d \geq i_2},
\end{split}
\]
where $\chi$ denotes an indicator function.
This is indeed the Right Hand Side equation \eqref{3.9}, as desired.
Note that in the last equality, we used Lemma \ref{lemma:CombiIncExc}, which we prove below. 

\end{proof}

\begin{lemma}\label{lemma:CombiIncExc}
    For any nonnegative integers $k,l,m$ such that $0 \leq l \leq k, \ 0 \leq m \leq k$,
    \begin{equation}\label{eq:BinomFormula}
    \sum_{j=0}^{\min\{k-l, m\} }(-1)^j \binom{k-l}{j}\binom{k-j}{k-m}
    = \binom{l}{l-m} \chi_{l \geq m} .
    \end{equation}
\end{lemma}

\begin{proof}
This can be proved by using a generating function. 
Looking at the coefficients of $x^m$ in 
\[
\begin{split}
&\sum_{j=0}^{k-l} \binom{k-l}{j} (-1)^j x^j (1+x)^{k-j}\\
&(=(1+x)^k \sum_{j=0}^{k-l} \binom{k-l}{j} \left(-\frac{x}{1+x}\right)^j
=(1+x)^k \left(1-\frac{x}{1+x}\right)^{k-l})\\
&=(1+x)^l
\end{split}
\]
enables us to get the desired equation.
\end{proof}

\begin{remark}
As in the proof proof of Lemma \ref{lemma:CombiIncExc}, there is also a combinatorial interpretation.
The left-hand side of (\ref{eq:BinomFormula}) can be interpreted as an inclusion-exclusion principle applied to subsets. 
Suppose that we partition a set of size $k$ into two subsets: $S$ (size $k-l$) and $T$ (size $l$). The sum counts the number of $m$-element subsets of $T$, adjusted by inclusion-exclusion to exclude overlaps with $S$.

\end{remark}

\section{Asymptotic behavior of Weingarten Functions}\label{sec:Asymptotics}

Let us use the formula of Theorem \ref{thm:CenteredWg2} for $\mathring{\Wg}_k(\sigma, \tau, N)$ in order to study the leading term $\widetilde{\mathring{\Wg}}_k(\sigma, \tau, N):=bN^{-a}$ of $\mathring{\Wg}_k(\sigma, \tau, N)$ as $N \to \infty $, i.e. the function which satisfies
\begin{equation}\label{eq:asym_CenteredWg}
\mathring{\Wg}_k(\sigma, \tau, N)
=\widetilde{\mathring{\Wg}}_k (\sigma, \tau, N) (1+O(N^{-1})). 
\end{equation}

\begin{proposition}\label{prop:asym_CenteredWg}
The leading term $\widetilde{\mathring{\Wg}}_k(\sigma, \tau, N)$ 
defined in (\ref{eq:asym_CenteredWg}) satisfies:
\begin{enumerate}
    \item if $d:=|D(\sigma \vee \tau)|$ is equal to $k$, which means that $\sigma=\tau=1_k$, 
\[
\widetilde{\mathring{\Wg}}_k(\sigma, \tau, N)= \tilde{a}_k (N),
\]
    \item if $d$ is even and $d \neq k$, 
\[
\widetilde{\mathring{\Wg}}_k(\sigma, \tau, N)=
\mu(\sigma \wedge \tau, \sigma)\mu(\sigma \wedge \tau, \tau) \{(d-1)!!\} N^{-(\# (\sigma \wedge \tau)+\tfrac{d}{2} )},
\]
    \item and if $d$ is odd and $d \neq k$,
\[
\widetilde{\mathring{\Wg}}_k(\sigma, \tau, N)=
\mu(\sigma \wedge \tau, \sigma)\mu(\sigma \wedge \tau, \tau) \left(\# (\sigma \wedge \tau)-\tfrac{d}{3}-\tfrac{2}{3}\right)(d!!) N^{-(\# (\sigma \wedge \tau)+\tfrac{d+1}{2})}.
\]
\end{enumerate}

\end{proposition}

\begin{proof}
This is shown directly by combining Theorem \ref{thm:CenteredWg2} with the asymptotics of $a_k(N)$ (Proposition \ref{prop:a_k}).

When $d$ is even or $\#(\sigma \wedge \tau) - |D(\sigma \vee \tau)|=0$ (i.e. \  $d=k$), in the summand of (\ref{eq:CenteredWg_improved}), only one term 
\[\mu(\sigma \wedge \tau, \sigma)\mu(\sigma \wedge \tau, \tau) N^{-(\#(\sigma \wedge \tau)-d)} a_{d}\]
is the leading term with the leading order $N^{-(\# (\sigma \wedge \tau)+\tfrac{d}{2} )}$. 

When $d$ is odd and not equal to $k$, two terms 
\[
\begin{split}
&\mu(\sigma \wedge \tau, \sigma)\mu(\sigma \wedge \tau, \tau) N^{-(\#(\sigma \wedge \tau)-d)} a_{d}, \\
&\quad \mbox{ and }\  
\mu(\sigma \wedge \tau, \sigma)\mu(\sigma \wedge \tau, \tau) (\#(\sigma \wedge \tau)-d)N^{-(\#(\sigma \wedge \tau)-d-1)} a_{d+1}
\end{split}
\]
have the leading order $N^{-(\# (\sigma \wedge \tau)+\tfrac{d+1}{2})}$.
\end{proof}

The above proposition clearly shows the difference in the decay speed as $N \to \infty$ of a noncentered Weingarten coefficient and the corresponding centered Weingarten coefficient. 
Compared to the fact that the asymptotics of $\Wg_k(\sigma, \tau, N)$ is $O(N^{-\# (\sigma \wedge \tau)})$ as we have seen in Corollary \ref{cor:Asym_WgPerm}, the leading order of $\mathring{\Wg}_k(\sigma, \tau, N)$ is $O(N^{-(\# (\sigma \wedge \tau)+\lceil\tfrac{d}{2} \rceil)} )$, where $d:=|D(\sigma \vee \tau)|$.

In particular, we get the following Corollary that explains very well the role of singletons, as singled out in 
\cite{B20_MR4203039},\cite{BC19_MR4024563}:

\begin{corollary}
The following estimates hold:

     $$\mathring{\Wg}_k(\sigma, \tau, N)=\Wg_k(\sigma, \tau, N)O(N^{-\lceil\tfrac{|D(\sigma \vee \tau)|}{2} \rceil})$$
and
$$\Wg_k(\sigma, \tau, N)=
\mathring{\Wg}_k(\sigma, \tau, N)O(N^{\lceil\tfrac{|D(\sigma \vee \tau)|}{2} \rceil})$$
     
\end{corollary}
Specifically, the exponent $\lceil\tfrac{|D(\sigma \vee \tau)|}{2} \rceil$ is consistent with the critical decay of the exponent $a_1$ in Proposition 11 of \cite{B20_MR4203039}.

\begin{remark}\label{rem:UniformBoundCenteredWg}
For $\sigma , \tau$ with the assumption that $k-\#(\sigma \wedge \tau)$ is bounded above by some universal constant (for example, in the case that $\sigma \wedge \tau = 0_k$), we can say that
if $N \gg k^3$ and $k \geq 1$, then
\begin{equation}\label{eq:UniformBoundCenteredWg}
\mathring{\Wg}_k(\sigma, \tau, N)
= \widetilde{\mathring{\Wg}}_k(\sigma, \tau, N) (1+ O(k^3/N)).
\end{equation}
This prompts two questions:

(1) Can we always reduce to the case where $k-\#(\sigma \wedge \tau)$ is bounded above?
The answer turns out to be yes, as we will see in the subsequent section \ref{sec:reduction}, where we exhibit a reduction of 
$\int_{g \in S_{N}} [g_{i_{1} j_{1}}] \cdots [g_{i_{k} j_{k}}] \d g $ 
to the case where $\Pi_{\mathbf{i}} \wedge \Pi_{\mathbf{j}} = 0_k$, 
so that only the coefficients
$\mathring{\Wg}_k(\sigma, \tau, N)$ with $\sigma \wedge \tau = 0_k$ contribute to the sum.

(2) What happens without an above bound for $k-\#(\sigma \wedge \tau)$?
In section \ref{sec:FailureUniformBound}, we show that  
without the assumption on $\sigma, \tau$ like above, 
uniform bounds can fail to hold.
\end{remark}

\section{Recursive reductions}\label{sec:reduction}
The goal of this section is to show that any integral of the type
$\int_{g \in S_{N}} [g_{i_{1} j_{1}}] \cdots [g_{i_{k} j_{k}}] \d g $.
can be transformed into a linear combination of integrals satisfying
$\Pi_{\mathbf{i}} \wedge \Pi_{\mathbf{j}} = 0_k$.
We start with the following lemma.

\begin{lemma}
For any permutation matrix $g \in S_N$ and any $1 \leq i,j \leq N$,
\begin{equation}
[g_{i j}]^l = \alpha_l [g_{i j}] + \beta_l \qquad \mbox{for \ }l \geq 1,
\end{equation}
where
\begin{equation}\label{eq:alpha_beta}
\begin{cases}
    \alpha_l = (\frac{N-1}{N})^l - (-\frac{1}{N})^l 
    \\
    \beta_l = \frac{1}{N} (\frac{N-1}{N})^l + \frac{N-1}{N} (-\frac{1}{N})^l 
\end{cases} \qquad \mbox{for \ }l \geq 1.
\end{equation}

\end{lemma}

\begin{proof}
This can be proved by induction on $l$ using the fact that
$g_{i j}^m = g_{i j}$ for all $m \geq 1$.
\end{proof}

\begin{remark}
    $\alpha_l, \beta_l$ can be uniformly estimated as follows: if $N \gg k$, then for any $k \geq l \geq 2$,
    \begin{equation}
    \alpha_l = 1+ O(k/N), \ 
    \beta_l = \frac{1}{N} (1+ O(k/N)).
    \end{equation}
    We note that $\alpha_1=1, \ \beta_1 = 0$.
\end{remark}

For $\Pi_{\mathbf{i}},\Pi_{\mathbf{j}} \in \mathcal{P}(k)$, when we define
$r:= \# (\Pi_{\mathbf{i}} \wedge \Pi_{\mathbf{j}})
$, we may assume 
$\Pi_{\mathbf{i}} \wedge \Pi_{\mathbf{j}} = \{B_1, \cdots , B_r\}$ where $B_1, \cdots , B_r$ are the blocks of $\Pi_{\mathbf{i}} \wedge \Pi_{\mathbf{j}}$ so that 
$\# B_1 \geq \cdots \geq \# B_r \geq 1$. 
(i.e. \ 
$l_1, l_2 \in B_s \Leftrightarrow i_{l_1}= i_{l_2} \& j_{l_1}=j_{l_2}$
)\\
Then we can naturally define new partitions 
$\Pi_{\mathbf{i}}^\prime := (i_{B_s})_{s=1}^r, \Pi_{\mathbf{j}}^\prime := (j_{B_s})_{s=1}^r \in \mathcal{P}(r)$,
which satisfy 
$\Pi_{\mathbf{i}}^\prime \wedge \Pi_{\mathbf{j}}^\prime = 0_r$
, and 
\[
\begin{split}
\int_{g \in S_{N}} \prod_{l=1}^k [g_{i_l j_l}] \d g 
&= \int_{g \in S_{N}} \prod_{s=1}^r [g_{i_{B_s} j_{B_s}}]^{\# B_s} \d g \\
&= \int_{g \in S_{N}} \prod_{s=1}^r ( \alpha_{\# B_s} [g_{i_{B_s} j_{B_s}}] + \beta_{\# B_s}) \d g \\
&= \sum_{M \subset [r]} (\prod_{s \in M} \alpha_{\# B_s}) (\prod_{s \notin M} \beta_{\# B_s})\int_{g \in S_N} \prod_{s \in M} [g_{i_{B_s} j_{B_s}}] \d g.
\end{split}
\]
Here, we note that
\[\int_{g \in S_N} \prod_{s \in M} [g_{i_{B_s} j_{B_s}}] \d g
=\sum_{\sigma, \tau \in \mathcal{P}(|M|)} \mathring{\Wg}_{|M|} (\sigma, \tau, N) \zeta(\sigma, \Pi^\prime_{\mathbf{i}}|_M) \zeta(\tau, \Pi^\prime_{\mathbf{j}}|_M),\]
and only the Weingarten coefficients $\mathring{\Wg}_{|M|}(\sigma, \tau, N)$ with
$\sigma \wedge \tau =0_{|M|}$ contribute to the sum since
$\Pi^\prime_{\mathbf{i}}|_M \wedge \Pi^\prime_{\mathbf{j}}|_M = 0_{|M|}
\ \ (\forall M \subset [r])$ as mentioned in Remark \ref{rem:UniformBoundCenteredWg}.

\section{Examples of the failure for the uniform bound of the Weingarten coefficients}\label{sec:FailureUniformBound}
As we saw from remark \ref{rem:UniformBoundCenteredWg}, 
for $\sigma , \tau$ under the assumption that $k-\#(\sigma \wedge \tau)$ is bounded by some universal constant in terms of $k$, we can obtain a uniform bound for $\mathring{\Wg}_k (\sigma, \tau, N)$, like (\ref{eq:UniformBoundCenteredWg}). 
However, without such an assumption on $\sigma, \tau$, we cannot obtain the corresponding uniform bounds. 
This is mainly because in formula (\ref{eq:CenteredWg_improved}) as $j$ increases over $\# (\sigma \wedge \tau) \leq j \leq k$, the contribution of \[\sum_{\substack{\pi \leq \sigma \wedge \tau \\ \#\pi=j }}^{} \mu(\pi, \sigma) \mu(\pi, \tau)\] increases exponentially with respect to $k$.
Here, we explain through an example why the uniform bound can not hold without an assumption on the maximum value of
$k-\#(\sigma \wedge \tau)$. For $\sigma =\tau= 1_k$, 
since 
\[\# \{\pi \leq 1_k(=\sigma \wedge\tau) \ | \ \# \pi= 2 (= \# (\sigma \wedge \tau) +1)\}
=2^{k-1} - 1,\]
by Theorem \ref{thm:CenteredWg2},
\[
\begin{split}
\mathring{\Wg}_k (\sigma, \tau, N)
&= \mathring{\Wg}_k (1_k, 1_k, N)\\
&\geq \sum_{j=1(=\#(\sigma \wedge \tau))}^k \left(\sum_{\substack{\pi \leq 1_k(=\sigma \wedge \tau) \\ \#\pi=j }}^{} \mu(\pi, 1_k) \mu(\pi, 1_k)\right) N^{-j} a_0\\
&\geq N^{-1}a_0 + \left(\sum_{\substack{\pi \leq 1_k(=\sigma \wedge \tau) \\ \#\pi=2 }}^{} \mu(\pi, 1_k) \mu(\pi, 1_k)\right)N^{-2}a_0\\
&=\left(N^{-1} + (2^{k-1} - 1)N^{-2}\right)a_0\\
&=\widetilde{\mathring{\Wg}}_k (1_k, 1_k, N)(1+(2^{k-1} - 1)N^{-1}).
\end{split}
\]
Here, we note that the leading term $\widetilde{\mathring{\Wg}}_k (1_k, 1_k, N)=N^{-1}$, and we observe that the coefficients of the lower order terms are exponential (in $k$) greater than that of the leading order term. This explains that it is not possible to get a uniform bound as in (\ref{eq:asym_CenteredWg}) 
in this case.
 
\section{Studying centered moments from centered Weingarten functions: examples}\label{sec:moment}
An initial motivation to study centered Weingarten calculus for random permutations was to study the centered moments of random permutations. While centered Weingarten calculus provides substantial insight, we show through examples that some phenomena occur that show that one can not yet fully capture the decay of some centered moments through Weingarten calculus. 

\begin{example}
    Consider the quantity, for $p \geq 2$
    \[
    \int_{g \in S_N} [g_{1 1}]^{p}[g_{2 2}]^{p}\cdots [g_{k k}]^{p} \d g.
    \]
    With the reduction in Section \ref{sec:reduction},
    \[
    \begin{split}
    &\int_{g \in S_N} [g_{1 1}]^{p}[g_{2 2}]^{p}\cdots [g_{k k}]^{p} \d g\\
    &=\int_{g \in S_N} (\alpha_{p}[g_{1 1}]+\beta_{p})(\alpha_{p}[g_{2 2}]+\beta_{p})\cdots (\alpha_{p}[g_{k k}]+\beta_{p}) \d g \\
    &=\sum_{l=0}^k \binom{k}{l} \alpha_p^l \beta_p^{k-l} \mathring{\Wg}_l (0_l,0_l,N),
    \end{split}
    \]
    where $\alpha_p \mbox{ and } \beta_p$ is defined as \eqref{eq:alpha_beta}.\\
    The leading term for this quantity is the term of $l=0$, \ $\beta_p^k a_0(N)=N^{-k} (1+O(pk/N)).$
    Moreover, by the uniform bound of $\{a_k(N)\}_k$, we can say that for $N \gg pk^3, k\geq 1$ 
    \[ 
    \int_{g \in S_N} [g_{1 1}]^p[g_{2 2}]^p\cdots [g_{k k}]^p \d g 
    =N^{-k} (1+O(\tfrac{pk^3}{N})).
    \]
    In this example, after expanding, given that $p\ge 2$, we see that the leading is the only term that does not involve a $g_{ii}$. In other words, the knowledge of centered Weingarten asymptotics do not appear in the leading order, they just serve to secure a uniform bound. 
\end{example}
\vspace{\baselineskip}
We describe another example.
\begin{example}
    A direct computation shows that the quantity
    \[
    \int_{g \in S_N} [g_{1 1}][g_{1 1}][g_{1 2}][g_{1 2}][g_{1 3}][g_{1 4}] \d g
    \]
    is equal to $\frac{2}{N^5}-\frac{5}{N^6}$.
    On the other hand, with the reduction,
    \begin{equation}\label{eq:loss-precision}
    \begin{split}
    &\int_{g \in S_N} [g_{1 1}][g_{1 1}][g_{1 2}][g_{1 2}][g_{1 3}][g_{1 4}] \d g\\
    &=\int_{g \in S_N} (\alpha_2[g_{1 1}]+\beta_2)(\alpha_2[g_{1 2}]+\beta_2)[g_{1 3}][g_{1 4}] \d g\\
    &=\alpha_2^2 \sum_{\sigma \leq 1_4} \mathring{\Wg}_4(\sigma, 0_4, N)
    +2 \alpha_2 \beta_2 \sum_{\sigma \leq 1_3} \mathring{\Wg}_3(\sigma, 0_3, N)
    +\beta_2^2 \sum_{\sigma \leq 1_2} \mathring{\Wg}_2(\sigma, 0_2, N),
    \end{split}
    \end{equation}
    where $\alpha_2=1-\tfrac{2}{N}
    , \ \beta_2=\tfrac{1}{N}-\tfrac{1}{N^2}
    $.\\
    With the fact that 
    $\mathring{\Wg}_l(\sigma, 0_l, N)=O(N^{-l})$
    as we saw in Section \ref{sec:Asymptotics}, 
    this yields that 
    \[
    \int_{g \in S_N} [g_{1 1}][g_{1 1}][g_{1 2}][g_{1 2}][g_{1 3}][g_{1 4}] \d g = O(N^{-4}).
    \]
    Given that the actually leading term is $2N^{-5}$, this estimate $O(N^{-4})$ carries less information, which means that in fact the terms with the order $N^{-4}$ vanish by the cancellation in the sum in \eqref{eq:loss-precision}.     
\end{example}

\bibliography{ref} 
\bibliographystyle{alpha}

\end{document}